\documentclass[12pt]{amsart}
\usepackage{amssymb}

\newtheorem{theorem}[equation]{Theorem}

\newtheorem{corollary}[equation]{Corollary}

\theoremstyle{definition}

\theoremstyle{remark}

\newtheorem{example}[equation]{Example}

\numberwithin{equation}{section}

\newcommand{\FF}{\mathbb{F}}
\newcommand{\ZZ}{\mathbb{Z}}

\newcommand{\NN}{\mathbb{N}}
\newcommand{\KK}{\mathbb{K}}

\newcommand{\be}{\mathbf{e}}

\newcommand{\cL}{\mathcal{L}}

\newcommand{\cN}{\mathcal{N}}

\newcommand{\cS}{\mathcal{S}}

\DeclareMathOperator{\Log}{Log}

\DeclareMathOperator{\Span}{Span}

\DeclareMathOperator{\sgn}{sgn}

\DeclareMathOperator{\trdeg}{tr.deg}

\newcommand{\oK}{\overline{K}}

\newcommand{\tpi}{\widetilde{\pi}}

\newcommand{\power}[2]{{#1 [\![ #2 ]\!]}}
\newcommand{\laurent}[2]{{#1 (\!( #2 )\!)}}

\begin{document}

\title[Values of Goss $L$-functions at $s=1$]{Algebraic independence of
values \\ of Goss $L$-functions at $s=1$}

\author{Brad A. Lutes}
\address{Department of Mathematics, Texas A{\&}M University, College Station,
TX 77843, USA}
\email{blutes@math.tamu.edu}

\author{Matthew A. Papanikolas}
\address{Department of Mathematics, Texas A{\&}M University, College Station,
TX 77843, USA}
\email{map@math.tamu.edu}

\dedicatory{In honor and memory of David R.~Hayes}

\thanks{This project was supported by NSF Grant DMS-0903838.}

\subjclass[2000]{Primary 11J93; Secondary 11G09, 11M38}

\date{May 31, 2011}

\begin{abstract}
  We investigate special values of Goss $L$-functions for Dirichlet
  characters at $s=1$ over rings of class number one and prove results
  on their transcendence and algebraic independence.
\end{abstract}

\keywords{Goss $L$-functions, special values, algebraic independence,
  Drinfeld logarithms, special polynomials}

\maketitle

\section{Introduction}
Given a ring of functions in one variable over a finite field that has class number one and a degree one rational place, we consider special values of Goss $L$-functions for Dirichlet characters at $s=1$.  Although we are
restricted to only finitely many possible rings, we obtain precise
formulas for the transcendence degrees of these special $L$-values as
the character varies.

We first establish some notation about rings of class number one.  We
let $q$ denote a fixed power of a prime number~$p$.  We let $A_0 :=
\FF_q[\theta]$.  For $j = 1, \dots, 4$, we let $A_j :=
\FF_q[\theta,\eta]/(f_j)$, where
\begin{equation} \label{E:fjdefs}
\begin{aligned}
f_1 &= \eta^2 - \theta^3 + \theta + 1 \in \FF_3[\theta,\eta],\\
f_2 &= \eta^2 + \eta + \theta^3 + \alpha \in \FF_4[\theta,\eta],\
\alpha \in \FF_4,\ \alpha^2 + \alpha + 1=0,\\
f_3 &= \eta^2 + \eta + \theta^3 + \theta + 1 \in \FF_2[\theta,\eta],\\
f_4 &= \eta^2 + \eta + \theta^5 + \theta^3 + 1 \in \FF_2[\theta,\eta].
\end{aligned}
\end{equation}
Even though $A_1, \dots, A_4$ are defined over specific finite fields,
when discussing these rings generically, we will use the convention
that `$\FF_q$' denotes the base field appropriate to the ring in
question.  Now for each $j$, we let $K_j$ be the fraction field of
$A_j$, and it is well known that each $K_j$ is the function field of a
smooth projective irreducible curve $X_j$ over $\FF_q$ such that $A_j$
is the ring of regular functions on $X_j$ away from a fixed closed
point $\infty$ on $X_j$ of degree $1$.  Most important for our
purposes is that each $A_j$ is a principal ideal domain, and they
represent a complete list of rings of algebraic functions over finite
fields with a degree one rational place that have class number one
\cite{LMQ75}.

Now let $A$ be one of $A_0, \dots, A_4$, with corresponding $f$, $K$, $X$, etc.~as above.  We fix a uniformizer $u := u_j \in K$ at $\infty$: $u_0 := 1/\theta$; for $j=1$, $2$, or $3$, $u_j :=
\theta/\eta$; and $u_4 := \theta^2/\eta$.  We let $K_\infty := \laurent{\FF_q}{u}$ be the completion of $K$ at $\infty$, we let $\KK$ be the completion of an algebraic closure of $K_\infty$, and we set $|\cdot|$ to be the absolute value on $\KK$ such that $|u|= 1/q$.  We define a sign function $\sgn: K_\infty^\times \to \FF_q^\times$ that is trivial on the $1$-units of $K_\infty$ and takes $\sgn(u)=1$, and we set
\[
  A_+ := \{ a \in A : \sgn(a) = 1 \}
\]
to be the monic elements of $A$.

Choose an irreducible polynomial $\wp \in A_+$ with residue field
$\FF_{\wp} := A/\wp$, and let $\chi : \FF_{\wp}^\times \to \FF_{\wp}^\times
\subseteq \KK^\times$
be a homomorphism.  We extend $\chi$ to a Dirichlet character $\chi: A
\to \FF_{\wp}$ by setting $\chi(a) = 0$ whenever $\wp \mid a$.
We then define values of the Goss $L$-function for $\chi$ at positive integers
to be
\[
  L(s,\chi) = \sum_{a \in A_+} \frac{\chi(a)}{a^s} \in
  \laurent{\FF_{\wp}}{u}, \quad s \in \NN.
\]
Goss extends $L(s,\chi)$ to a much larger analytic domain \cite[Ch.~8]{Goss}, but we
are only concerned here with the sums defined above with $s \in \NN$.

In \cite{And96} Anderson related the values $L(1,\chi)$ for $A_0 =
\FF_q[\theta]$ to Carlitz logarithms of algebraic numbers.  Anderson's
results were extended by the first author~\cite{LutesThesis} to $A_1,
\dots, A_4$, where $L(1,\chi)$ was expressed in terms of Drinfeld
logarithms on $\sgn$-normalized Drinfeld-Hayes modules (see
\S\ref{S:logalg}).  Combining these formulas with algebraic
independence results on Carlitz and Drinfeld logarithms due to Chang
and the second author \cite{CP,Papanikolas08}, we prove the following
formula for the transcendence degree of the values $L(1,\chi)$, as
$\chi$ varies, over the algebraic closure $\oK$ of $K$ contained in
$\KK$.

\begin{theorem} \label{T:Main}
Let $\wp \in A_+$ be an irreducible polynomial of degree $d$, and
let $\Xi_{\wp}$ be the group of all Dirichlet characters modulo $\wp$
on $A$.  For each $\chi \in \Xi_{\wp}$, $L(1,\chi)$ is transcendental
over~$\oK$.  Furthermore,
\[
  \trdeg_{\oK}\ \oK \bigl(L(1,\chi) : \chi \in \Xi_{\wp} \bigr) =
  \frac{(q^d-1)(q-2)}{q-1} + 1.
\]
\end{theorem}

When $\chi$ is the trivial character, then up to the Euler factor at
$\wp$, $L(1,\chi)$ is essentially the zeta value
\[
  \zeta_A(1) = \sum_{a \in A_+} \frac{1}{a}.
\]
These values were proved to be transcendental by
Thakur~\cite{Thakur92} and Yu~\cite{Yu86}.  Some cases of the
transcendence of $L(1,\chi)$ were also established by
Damamme~\cite{Damamme99}.  In the classical case over number fields,
results in these directions have recently been obtained by Murty and
Murty \cite{MurtyMurty1}, \cite{MurtyMurty2}, assuming Schanuel's
conjecture.  It would be quite interesting to place the results here
in the context of recent work of Pellarin~\cite{Pellarin} or
Taelman~\cite{Taelman}, but we have not pursued these lines of inquiry yet.

\section{Preliminaries on Drinfeld-Hayes modules}
\label{S:prelims}

Having fixed $A$, $K$, and $\sgn : K_\infty^\times \to \FF_q^\times$, we know
from the theory of Drinfeld modules that there is a unique
`$\sgn$-normalized' Drinfeld $A$-module of rank~$1$ defined over $K$,
called the Drinfeld-Hayes module (see \cite[Ch.~VII]{Goss}, \cite[\S
3]{Hayes79}, \cite[Ch.~3]{Thakur}).  We let $K\{\tau\}$ be the ring of
twisted polynomials in $\tau$ over $K$, where multiplication is
governed by the rule $\tau c = c^q\tau$ for $c \in K$.  The
Drinfeld-Hayes module $\rho$ is defined to be the $\FF_q$-algebra
homomorphism
\[
  \rho : A \to K\{\tau\},
\]
such that (1) the constant term in $\tau$ of $\rho_a$ is $a$, (2)
$\deg_{\tau}(\rho_a) = \deg(a)$, and (3) the top coefficient of
$\rho_a$ is $\sgn(a)$.  Since $K\{\tau\}$ can be identified with a subring of the ring of $\FF_q$-linear endomorphisms of $K$, we see that $\rho$ induces an $A$-module structure on any extension field $L$ of $K$, denoted by $a \cdot x := \rho_a(x)$, $a \in A$, $x \in L$.
We let $(\rho,L)$ denote $L$ with the $A$-module structure induced by $\rho$.

The existence and uniqueness of such a Drinfeld module are by no means guaranteed, and its construction is fundamental to explicit class field theory for $K$.  Furthermore, for more general rings $A$ (those without class number one), such a $\sgn$-normalized module can be defined only over the Hilbert class field of $K$.  See Hayes \cite{Hayes79}, \cite{Hayes92} for additional information and connections with class field theory.

For each $j = 0, \dots, 4$, the Drinfeld-Hayes
module $\rho^j : A_j \to K_j\{\tau\}$, can be defined by its action on generators of $A_j$ from the defining polynomials in \eqref{E:fjdefs}.  Precise definitions of each $\rho^j$ can be found in \cite[\S 11]{Hayes79} (and also \cite[\S 7.11]{Goss} or \cite{Thakur93}).
When $A=A_0$, then $\rho^0$ is the Carlitz module $\rho^0 : \FF_q[\theta] \to K_0\{\tau\}$, defined by
\begin{equation} \label{E:Carlitz}
  \rho^0_\theta = \theta + \tau.
\end{equation}
When $A=A_1$, $\rho^1 : \FF_3[\theta,\eta] \to K_1\{\tau\}$ is defined by
\begin{equation} \label{E:rho1}
\begin{aligned}
\rho^1_\theta &= \theta + (\eta^3+\eta)\tau + \tau^2, \\
\rho^1_\eta &= \eta + (\eta^4-\eta^2)\tau + (\eta^9 +\eta^3+\eta)\tau^2 + \tau^3.
\end{aligned}
\end{equation}
Since we will not appeal to the precise definitions of $\rho^2$, $\rho^3$, and $\rho^4$ in what follows, we direct the reader to the above references for more details about them.

Now there is an exponential function $\exp_{\rho} : \KK \to \KK$ associated to each Drinfeld module $\rho = \rho^j$, which is given by an $\FF_q$-linear power series
\[
  \exp_\rho(z) = z + \sum_{i \geq 1} e_{\rho,i} z^{q^i} \in \power{K}{z}.
\]
The coefficients of $\exp_\rho(z)$ are determined uniquely by the conditions that
\begin{equation} \label{E:expfneq}
  \exp_\rho(az) = \rho_a \bigl( \exp_\rho(z) \bigr), \quad a \in A,
\end{equation}
and the function $\exp_\rho(z)$ is entire, $\FF_q$-linear, and surjective on $\KK$.  There is a period $\tpi := \tpi_j$ in a finite extension of $K_\infty$ such that we have an exact sequence of $A$-modules
\[
  0 \to A\tpi \to \KK \xrightarrow{\exp_\rho} (\rho, \KK) \to 0,
\]
which uniformizes $\rho$ as the quotient of $\KK$ by a discrete $A$-lattice of rank~$1$.  The inverse of $\exp_\rho(z)$ is given by a power series
\[
  \log_\rho(z) = z + \sum_{i \geq 1} l_{\rho,i} z^{q^i} \in \power{K}{z},
\]
which is convergent for $|z| < |\tpi|$.  For convenience, we record that
\begin{equation} \label{E:pibounds}
|\tpi_0| =  q^{q/(q-1)}, \quad |\tpi_1| = 3^{-3/2}, \quad |\tpi_2| = 4^{-8/3}, \quad
|\tpi_3| = 1, \quad |\tpi_4| = 2^{-4},
\end{equation}
each of which follows from \cite[Cor.~7.10.11]{Goss}.  (See \cite[\S IV.E]{LutesThesis} for details.)

\section{Log-algebraicity formulas and the proof of Theorem~\ref{T:Main}} \label{S:logalg}
In \cite{And96}, Anderson proved a log-algebraic power series identity for twisted harmonic sums over rings of algebraic functions, valid for any $\sgn$-normalized rank $1$ Drinfeld-Hayes module, even over rings of class number $> 1$.  The results there were based on earlier work of Anderson~\cite{And94}, which provided similar power series identities without twisting and which themselves were inspired by results of Thakur~\cite{Thakur92} on special zeta values.  The connections between log-algebraicity identities and special $L$-values $L(1,\chi)$ were first introduced in \cite{And96} for the case $A=\FF_q[\theta]$ and were investigated further in \cite{LutesThesis} for general $A$, in particular for $A_1, \dots, A_4$.

In the context of the Drinfeld-Hayes modules $\rho^0, \dots, \rho^4$ already defined, Anderson's main theorem is the following.

\begin{theorem}[{Anderson~\cite[Thm.~3]{And96}}]
Let $\rho$ be one of the Drinfeld-Hayes modules $\rho^0, \dots, \rho^4$ defined in \S\ref{S:prelims}.  For every $\beta$ in the polynomial ring $A[x]$, the power series
\[
  \exp_\rho \Biggl( \sum_{a \in A_{+}} \frac{ \beta(\rho_a(x))}{a} z^{q^{\deg a}} \Biggr) \in
  \power{K[x]}{z}
\]
is in fact in $A[x,z]$.
\end{theorem}

Taking $\beta = x^m$, $m \geq 0$, and $0 \leq j \leq 4$, we define \emph{special polynomials}
\begin{equation} \label{E:Smj}
  S_m^j(x,z) := S_m(x,z) = \exp_\rho \Biggl( \sum_{a \in A_{+}} \frac{ \rho_a(x)^m}{a} z^{q^{\deg a}} \Biggr) \in A[x,z].
\end{equation}
Anderson worked out extensive properties of special polynomials in the Carlitz module ($\rho =\rho^0$) case \cite[Prop.~8]{And96}, and he provided tables of examples \cite[\S 4.3]{And96}.
The first author extended Anderson's results on special polynomials to the cases of $\rho^1, \dots, \rho^4$ \cite[Prop.~IV.17]{LutesThesis}.  For some specific examples, see \S\ref{S:examples}.

Let $A$ be one of $A_0, \dots, A_4$ with corresponding Drinfeld-Hayes module $\rho$.  We set
\[
  \be_\rho(z) := \exp_\rho(\tpi z).
\]
Fixing a prime $\wp \in A_+$ of degree $d$, the $\wp$-torsion elements $\rho[\wp]$ of $\rho$ are given by all $\be_\rho ( b/\wp)$, $b \in A$.  Of particular interest for us are the specializations
\[
  S_m(\be_\rho(b/\wp),1) \in A[\be_\rho(b/\wp)], \quad b \in A,
\]
which lie in the $\wp$-th Carlitz cyclotomic extension $K':=K(\be_\rho(1/\wp))$.  If we let
\[
  \lambda_{\rho,m}(z) := \sum_{a \in A_{+}} \frac{ \be_\rho(az)^m }{a}, \quad m \geq 0,
\]
then by \eqref{E:expfneq} and \eqref{E:Smj} we have
\begin{equation} \label{E:explambda}
  \exp_\rho \bigl( \lambda_{\rho,m}(b/\wp) \bigr) = S_m(\be_\rho(b/\wp),1),
\end{equation}
and thus each $\lambda_{\rho,m}(b/\wp)$ is the logarithm of a $\oK$-valued point on $\rho$.

We let $\cL_\rho(\wp) = \Span_A \bigl( \lambda_{\rho,m}(b/\wp) : m \geq 0, b \in A \bigr)$ be the $A$-linear span of the elements $\lambda_{\rho,m}(b/\wp)$, and then the module $\cS_\rho(\wp)$ of \emph{special points of $\rho$} is defined to be the image of $\cL_\rho(\wp)$ under $\exp_\rho$.  Thus $\cS_\rho(\wp)$ is the $A$-submodule of $(\rho,K')$ generated by the $S_m(\be_\rho(b/\wp),1)$:
\[
  \cS_\rho(\wp) := \Span_{\rho(A)} \bigl( S_m(\be_\rho(b/\wp),1) : m \geq 0, b \in A \bigr).
\]
It is a fundamental result of Anderson that this module has finite rank for the Carlitz module, which was extended by the first author to the cases $\rho^1, \dots, \rho^4$.

\begin{theorem}[{\cite[Prop.~9 \& Thm.~4]{And96}; \cite[Prop.~IV.23 \& Thm.~VI.4]{LutesThesis}}]
\label{T:Arank}
Let $\cN := \{ 1\} \cup \{ m \in \ZZ : 1 < m \leq q^d-1,\ m \not\equiv 1\ (\mathrm{mod}\ q-1) \}$.
\begin{enumerate}
\item[(a)] We have $\Span_K \bigl( \cL_{\rho}(\wp) \bigr) = \Span_K \bigl( \lambda_{\rho,m}(1/\wp): m \in \cN \bigr)$.
\item[(b)] The $A$-module $\cS_\rho(\wp)$ is finitely generated of rank $\#\cN - 1 = (q^d-1)(q-2)/(q-1)$.
\end{enumerate}
\end{theorem}

We note that for (a) and (b) in the above theorem to be true there must be an element of $\cL_\rho(\wp)$, or at least an $A$-linear combination of elements of $\cL_\rho(\wp)$, whose exponential is torsion on $\rho$.  This is revealed in the proof to be
\begin{equation} \label{E:lambda1}
  \lambda_{\rho,1}(1/\wp) = \tpi/\wp, \quad S_1(\be_\rho(1/\wp),1) = \be_\rho(1/\wp).
\end{equation}
See the proofs of \cite[Prop.~8(VII)]{And96}, with $m=1$, or \cite[Thm.~VI.4]{LutesThesis}.

\begin{corollary} \label{C:lambdaLinInd}
The set $\{ \lambda_{\rho,m}(1/\wp) : m \in \cN\}$ is linearly independent over $K$.
\end{corollary}

\begin{proof}
Let $W = \Span_K \bigl( \cL_\rho(\wp)\bigr) \subseteq \KK$.  Since $\cS_{\rho}(\wp) = \exp_{\rho}(\cL_\rho(\wp))$ has $A$-rank $\#\cN - 1$, it follows that
\[
  \dim_K W = (\#\cN -1) + \dim_K(W \cap K\tpi).
\]
By \eqref{E:lambda1} we see that the second term is equal to $1$.
\end{proof}

One can also uniquely define \emph{dual coefficients} $\be_{\rho,m}^*(a) \in \oK$, for $a \in A$, $\wp \nmid a$, $m \geq 1$, by requiring, for $a$ and $b$ relatively prime to $\wp$,
\begin{equation} \label{E:estar}
  \sum_{m=1}^{q^{d}-1} \be_{\rho,m}^*(a) \be_\rho(b/\wp)^m = \wp \cdot \delta_{ab} =
  \begin{cases}
  \wp & \textnormal{if $a \equiv b \pmod{\wp}$,} \\
  0 & \textnormal{otherwise.}
  \end{cases}
\end{equation}
See \cite[Prop.~10]{And96} and
\cite[\S V.D]{LutesThesis} for details; results of Feng~\cite{Feng97} and Zhao~\cite{Zhao97} provide additional information on computing dual coefficients and associated root numbers in terms of Gauss sums in the case $\rho =\rho^0$.  The dual coefficients are defined precisely to obtain a formula for the partial zeta function,
\begin{equation} \label{E:partzeta}
  \sum_{\substack{ n \in A_+ \\ bn\, \equiv\, a\, (\mathrm{mod}\,\wp)}} \frac{1}{n}= \frac{1}{\wp} \sum_{m=1}^{q^d-1} \be_{\rho,m}^*(a) \lambda_{\rho,m}(b/\wp),
\end{equation}
for any $a$, $b \in A$ relatively prime to $\wp$.  This formula is obtained in a straightforward manner by substituting in the defining series for $\lambda_{\rho,m}(b/\wp)$ and simplifying the resulting expression (see \cite[\S 4.7]{And96} or \cite[\S V.E]{LutesThesis}).  Now for any Dirichlet character $\chi \in \Xi_{\wp}$, if we take $b=1$ in \eqref{E:partzeta}, multiply by $\chi(a)$, and take the sum over $a$, we obtain
\begin{equation} \label{E:L1chisum}
  L(1,\chi) = \sum_{m=1}^{q^d-1} \Biggl( \frac{1}{\wp} \sum_{a \in \FF_\wp^{\times}} \chi(a) \be_{\rho,m}^*(a) \Biggr) \lambda_{\rho,m}(1/\wp),
\end{equation}
where the inner sum runs over $a \in A$ representing classes in $\FF_{\wp}^\times$.  The inner sums are the root numbers.  Using the orthogonality relations for characters and \eqref{E:estar}, we find also that
\begin{equation} \label{E:lambdasum}
  \lambda_{\rho,m}(1/\wp) =  \sum_{\chi \in \Xi_{\wp}} \Biggl( \sum_{a \in \FF_\wp^\times}
  \chi^{-1}(a) \be_\rho(a/\wp)^m \Biggr) L(1,\chi).
\end{equation}
We can now prove our main result.

\begin{proof}[Proof of Theorem~\ref{T:Main}]
Because $\rho$ is a rank $1$ Drinfeld $A$-module, its endomorphism ring is simply $A$.  Therefore, by \cite[Thm.~4.3]{Yu97}, we see that $\oK$-linear combinations of logarithms of $\oK$-valued points on $\rho$ are either $0$ or transcendental.  Applying this to the sum in \eqref{E:L1chisum}, each $L(1,\chi)$ is transcendental over $\oK$, since $L(1,\chi) \neq 0$ (e.g.\ $|L(1,\chi)| = 1$).

Now by \eqref{E:L1chisum} and \eqref{E:lambdasum} we see that
\[
  \Span_{\oK} \bigl( L(1,\chi) : \chi \in \Xi_\wp \bigr) = \Span_{\oK} \bigl(
  \lambda_{\rho,m}(1/\wp) : 1 \leq m \leq q^d-1 \bigr),
\]
and thus it suffices to calculate the transcendence degree of $\oK\bigl( \lambda_{\rho,m}(1/\wp) : 1 \leq m \leq q^d-1 \bigr)$ over $\oK$.  Now from \eqref{E:explambda}, $\exp_\rho(\lambda_{\rho,m}(1/\wp)) \in \oK$ for each $m$.  Also, by Theorem~\ref{T:Arank} and Corollary~\ref{C:lambdaLinInd},  we see that $\{ \lambda_{\rho,m}(1/\wp) : m \in \cN \}$ is a $K$-basis of $\Span_K\bigl( \cL_{\rho}(\wp) \bigr)$.  Thus
\[
  \oK \bigl(\lambda_{\rho,m}(1/\wp) : m \in \cN \bigr) = \oK \bigl( \lambda_{\rho,m}(1/\wp) : 1 \leq m \leq q^d-1 \bigr) = \oK\bigl( \cL_{\rho}(\wp) \bigr).
\]
Finally, by \cite[Thm.~1.1.1]{CP} and \cite[Thm.~1.2.6]{Papanikolas08}, it follows that
\[
  \trdeg_{\oK}\ \oK \bigl(\lambda_{\rho,m}(1/\wp) : m \in \cN \bigr) = \#\cN =
  \frac{(q^d-1)(q-2)}{q-1} + 1.
\]
\end{proof}

\section{Examples} \label{S:examples}
We now calculate some examples to elaborate on Theorem~\ref{T:Main}.  Although formulas such as \eqref{E:L1chisum} work well for theoretical purposes, they often contain hidden cancelations, and sometimes $L$-values can be evaluated somewhat more directly.

\begin{example}
We begin by letting $\rho = \rho^0$ be the Carlitz module (see \eqref{E:Carlitz}) over $\FF_3[\theta]$ and taking $\wp = \theta$. The character group $\Xi_\theta$ is of order $2$, generated by the character $\chi : \FF_3[\theta] \to \FF_3$ defined by $\chi(a) = a(0)$.  Referring to \cite[\S 4.3]{And96}, we see that the first special polynomial for $\rho$ is
$S_1^0(x,z) = xz$.  Therefore by \eqref{E:explambda},
\[
  \exp_\rho \Biggl( \sum_{a \in \FF_3[\theta]_+} \frac{\be_\rho(a/\theta)}{a} \Biggr) = \be_\rho(1/\theta).
\]
Now standard facts about the Carlitz module \cite[Ch.~3]{Goss} imply that for all $a \in \FF_3[\theta]$
\[
  \rho_a(\be_\rho(a/\theta)) = \chi(a) \be_\rho(a/\theta),
\]
and also $\be_\rho(1/\theta) = \sqrt{-\theta}$.  Thus,
\[
  \exp_\rho \bigl( \sqrt{-\theta} L(1,\chi) \bigr) = \sqrt{-\theta}.
\]
Since $\sqrt{-\theta}$ is $\theta$-torsion, it follows that
\begin{equation} \label{E:L1chi1}
  \sqrt{-\theta} L(1,\chi) \in \FF_3[\theta]\cdot \frac{\tpi_0}{\theta}.
\end{equation}
Recalling that the Carlitz period over $\FF_3(\theta)$ has the expression
\[
\tpi_0 = \theta \sqrt{-\theta} \prod_{i=1}^\infty \biggl( 1 - \frac{1}{\theta^{3^i-1}}
\biggr)^{-1},
\]
(see \cite[\S 2.5]{Thakur}), it follows that the $\FF_3[\theta]$-multiple in \eqref{E:L1chi1} is simply $1$: indeed, the right-hand side is a discrete set, and so this can be checked after a finite amount of computation.  Therefore,
\begin{equation}
  L(1,\chi) = \frac{\tpi_0}{\theta\sqrt{-\theta}}.
\end{equation}
As $\chi^2$ is the trivial character modulo $\theta$, we observe (see \cite[\S 5.9]{Thakur}) that
\begin{equation} \label{E:L1epsilon}
  L(1,\chi^2) = \biggl(1-\frac{1}{\theta} \biggr) \log_\rho(1),
\end{equation}
and since $1$ is not a torsion point on $\rho$, we see that $\oK(L(1,\chi), L(1,\chi^2)) = \oK(\tpi_0, \log_\rho(1))$, and moreover
$L(1,\chi)$ and $L(1,\chi^2)$ are algebraically independent by \cite[Thm.~1.2.6]{Papanikolas08}.
\end{example}

\begin{example}
We again consider the Carlitz module case $\rho = \rho^0$ over
$\FF_3[\theta]$ but now take $\wp = \theta^2+1$; however, in the interests of space we must omit many of the details.  We let $\chi :
\FF_3[\theta] \to \FF_9$ be the character defined by $\chi(a) =
a(\sqrt{-1})$, which has order $8$.  For $\alpha \in \laurent{\FF_9}{1/\theta}$, we let $\alpha \mapsto \overline{\alpha}$ denote the canonical automorphism that sends $\sqrt{-1}$ to $-\sqrt{-1}$ and fixes $\theta$.  Thus for all $j$, $\overline{L(1,\chi^j)} = L(1,\chi^{3j})$.
We let
\[
  \zeta_1 := \be_\rho\biggl( \frac{1}{\theta^2+1} \biggr), \quad
  \zeta_2 := \be_\rho\biggl( \frac{\theta}{\theta^2+1} \biggr),
\]
and then $\rho[\theta^2+1] = \FF_3\zeta_1 + \FF_3\zeta_2$.  Now using that $S_1^0(x,z) = xz$, we see
\begin{equation} \label{E:expt21}
  \exp_\rho \Biggl( \sum_{a \in \FF_3[\theta]_+}
  \frac{\be_\rho(a/(\theta^2+1))}{a} \Biggr) = \zeta_1.
\end{equation}
A straightforward calculation reveals that for all $a \in \FF_3[\theta]$,
\[
  \be_{\rho}\biggl( \frac{a}{\theta^2+1} \biggr) = -(\chi(a) + \chi^3(a))
  \zeta_1 + \sqrt{-1}(\chi(a)-\chi^3(a)) \zeta_2.
\]
Therefore,
\[
  \sum_{a \in \FF_3[\theta]_{+}} \frac{\be_\rho(a/(\theta^2+1))}{a} =
  ( -\zeta_1 + \sqrt{-1}\,\zeta_2 )\cdot L(1,\chi)
  - (\zeta_1 + \sqrt{-1}\,\zeta_2 )\cdot L(1,\chi^3),
\]
and \eqref{E:expt21} implies that the right-hand side is in
$\FF_3[\theta]\cdot \tpi_0/(\theta^2+1)$.  Again a computation reveals
that this multiple is exactly $1$, and thus
\begin{equation} \label{E:chichi3one}
  ( -\zeta_1 + \sqrt{-1}\,\zeta_2 )\cdot L(1,\chi)
  - ( \zeta_1 +\sqrt{-1}\,\zeta_2 )\cdot L(1,\chi^3) = \frac{\tpi_0}{\theta^2+1}.
\end{equation}
Similarly, for $a \in \FF_3[\theta]$,
\[
  \be_{\rho}\biggl( \frac{a\theta}{\theta^2+1} \biggr) =
  -\sqrt{-1}(\chi(a) - \chi^3(a)) \zeta_1
   - (\chi(a)+\chi^3(a)) \zeta_2,
\]
and we find, after summing over all $a \in \FF_3[\theta]_+$, that
\begin{equation} \label{E:chichi3two}
  ( -\sqrt{-1}\,\zeta_1 - \zeta_2) \cdot L(1,\chi)
  + ( \sqrt{-1}\,\zeta_1 - \zeta_2 ) \cdot L(1,\chi^3) =
  \frac{\theta \tpi_0}{\theta^2+1}.
\end{equation}
Combining \eqref{E:chichi3one} and \eqref{E:chichi3two}, we obtain
\begin{align} \label{E:L1chichi3}
L(1,\chi) &= \frac{\tpi_0}{(\zeta_1 - \sqrt{-1}\,\zeta_2)(1 +
\sqrt{-1}\, \theta)}, \\
L(1,\chi^3) &= \frac{\tpi_0}{(\zeta_1+\sqrt{-1}\,\zeta_2)(1 -
\sqrt{-1}\, \theta)}, \quad \textnormal{($= \overline{L(1,\chi)}$).}
\end{align}
In a similar manner, using also that $\zeta_1^2 + \zeta_2^2 = \zeta_1\zeta_2(\zeta_1^2-\zeta_2^2) = -\sqrt{\theta^2+1}$ (taking the positive square root of $\theta^2+1$ with respect to $\sgn$),
\begin{multline}
L(1,\chi^2) = \frac{1}{\sqrt{\theta^2+1}} \left[ -\frac{\zeta_1^3\zeta_2^3}{\theta^2+1} \log_\rho(\zeta_1^2) + \frac{1}{\zeta_1\zeta_2} \log_\rho(\zeta_2^2) + \frac{\zeta_1\zeta_2}{\sqrt{\theta^2+1}} \log_\rho(\zeta_1^4-\zeta_1^6) \right] \\
{}+\sqrt{-1} \left[ -\frac{1}{\zeta_1^3\zeta_2^3} \log_\rho(\zeta_1^2) + \frac{\zeta_1 \zeta_2}{\theta^2+1} \log_\rho(\zeta_2^2) - \frac{1}{\zeta_1\zeta_2\sqrt{\theta^2+1}}
\log_\rho( \zeta_1^4-\zeta_1^6) \right],
\end{multline}
and $L(1,\chi^6) = \overline{L(1,\chi^2)}$.  We observe that $|\zeta_1| = 3^{-1/2}$ and $|\zeta_2| = 3^{1/2}$, and so $\zeta_1^2$, $\zeta_2^2$, and $\zeta_1^4-\zeta_1^6$ are within the radius of convergence of $\log_\rho(z)$ as given in \eqref{E:pibounds}.  Furthermore,
\begin{equation}
  L(1,\chi^4) = \frac{\log_\rho(\sqrt{\theta^2+1})}{\sqrt{\theta^2+1}}.
\end{equation}
Also,
\begin{equation}
  L(1,\chi^5) = \frac{\tpi_0}{(\sqrt{-1}\,\zeta_1 + \zeta_2) \sqrt{\theta^2+1}},
\end{equation}
and $L(1,\chi^7) = \overline{L(1,\chi^5)}$.  Finally, since $\chi^8$ is the trivial character, as in the previous example,
\begin{equation} \label{E:L1chi8}
  L(1,\chi^8) = \left( 1 - \frac{1}{\theta^2+1} \right) \log_\rho(1).
\end{equation}
Thus combining \eqref{E:L1chichi3}--\eqref{E:L1chi8}, we see that
\begin{align*}
  \oK\bigl( L(1, \chi^j) : 1\leq j \leq 8 \bigr) &= \oK\bigl( L( 1, \chi), L(1,\chi^2), L(1,\chi^4), L(1,\chi^6), L(1,\chi^8) \bigr) \\
  &= \oK \bigl(\tpi_0, \log_\rho(1), \log_\rho(\zeta_1^2), \log_\rho(\zeta_2^2), \log_\rho(\zeta_1^4-\zeta_1^6) \bigr).
\end{align*}
By an analysis similar to the one in the proof of Theorem~\ref{T:Arank}, these logarithms can be shown to be linearly independent over $\FF_3(\theta)$ and thus are algebraically independent over $\oK$ by \cite[Thm.~1.2.6]{Papanikolas08}.
\end{example}

\begin{example}
Let $\rho = \rho^1$ be the Drinfeld-Hayes module for $A = A_1 = \FF_3[\theta,\eta]$ defined in \eqref{E:rho1}.  We let $\chi : A \to \FF_9$ be the character modulo $\theta$ of order $8$ defined by $\chi(a) := a(0,\sqrt{-1})$.  Here we present formulas for $L(1,\chi^j)$, $1\leq j \leq 8$, which take values in $\laurent{\FF_9}{\theta/\eta} = \laurent{\FF_9}{1/\sqrt{\theta}}$.  As in the previous example $\overline{L(1,\chi^j)} = L(1,\chi^{3j})$, where $\alpha \mapsto \overline{\alpha}$ denotes the canonical automorphism of $\laurent{\FF_9}{1/\sqrt{\theta}}$ such that $\sqrt{-1} \mapsto -\sqrt{-1}$ and $\sqrt{\theta}$ is fixed.  As expected, we find that
\[
    \oK\bigl( L(1, \chi^j) : 1\leq j \leq 8 \bigr) = \oK\bigl( L( 1, \chi), L(1,\chi^2), L(1,\chi^4), L(1,\chi^6), L(1,\chi^8) \bigr),
\]
and that the terms on the right form a transcendence basis by Theorem~\ref{T:Main} via \cite[Thm.~1.1.1]{CP}.  One goal of this example is to make these $L$-values more explicit.  The calculations are more involved but similar in spirit to the previous examples, and again space limitations force us to omit many of the details.  Furthermore, because of computational difficulties (see \eqref{E:S1polys}), we were unable to calculate exact expressions for $L(1,\chi^2)$ and $L(1,\chi^6) = \overline{L(1,\chi^2)}$.

We calculate some preliminary data.  Using \cite[Eq.~(7.10.6)]{Goss} the period $\tpi_1$ of $\rho$ can be computed by taking the product
\begin{multline*}
  \Pi_1 = \theta^{-9}(\eta^6 + \eta^4 + \eta^2) \\ {}\cdot \prod_{a \in A_+}
   \left[ \left( 1- \frac{1}{a^2\theta^2} \right)\left(1 - \frac{\eta^2}{a^2\theta^2} \right)
   \left( 1 - \frac{(\eta+1)^2}{a^2\theta^2} \right) \left( 1 - \frac{(\eta-1)^2}{a^2\theta^2}
   \right) \right]^2,
\end{multline*}
and then extracting an $8$-th root,
\[
  \tpi_1 = \sqrt{-1}\cdot \Pi_1^{1/8},
\]
where the root of $\Pi_1$ is chosen to have $\sgn$ value $1$.  We let
\[
  \xi_1 := \be_\rho \biggl( \frac{1}{\theta} \biggr), \quad
  \xi_2 := \be_\rho \biggl( \frac{\eta}{\theta} \biggr),
\]
and thus $\rho[\theta] = \FF_3 \xi_1 + \FF_3\xi_2$.  We find that $|\xi_1| = 3^{-7/2}$ and $|\xi_2| = 3^{3/2}$, and furthermore
\[
  \xi_1\xi_2 (\xi_1^2 - \xi_2^2) = \sqrt{\theta}, \quad \xi_1^2 + \xi_2^2 = -(\theta+1)\sqrt{\theta},
  \quad \xi_1^4 + \xi_2^4 = \eta(\theta-1)\sqrt{\theta}.
\]
Finally we compute the special polynomials (see \cite[p.~492]{And94} and \cite[Ch.~VII]{LutesThesis}),
\begin{equation} \label{E:S1polys}
\begin{aligned}
  S_0^1(x,z) &= z + \eta z^3 + z^9,\\
  S_1^1(x,z) &= xz + \eta x^3z^3 - \eta x^3z^9 + x^9z^9 - x^9z^{27},\\
  S_2^1(x,z) &=
  \begin{aligned}[t]
  x^2z &+ \eta x^6z^3 - (\eta^4+\eta^2)x^6 z^9 + \eta x^{12}z^9 + x^{18}z^9\\
    &{}- (\eta^9 - \eta^3 - \eta) x^{18}z^{27} + x^{36}z^{27} + x^{54}z^{81}.
  \end{aligned}
\end{aligned}
\end{equation}
In spite of our best efforts we were unable to compute $S_4^1(x,z)$, which would be necessary for finding an exact expression for $L(1,\chi^2)$, but we obtained the partial result,
\begin{align*}
S_4^1(x,z) = \ &x^4z + \eta x^{12}z^3 - (\eta^3+\eta)x^6z^9 - (\eta^{10}+\eta^4 + 1)x^{12}z^9
  - (\eta^7 - \eta^5 + \eta^3) x^{18}z^9 \\
& {}- \eta x^{30}z^9 + x^{36}z^9 + (\eta^{16} - \eta^{14} + \eta^{12} - \eta^8 + \eta^4)
x^{18}z^{27} \\
& {}- (\eta^{27} + \eta^9 + \eta^7 - \eta^5 + \eta^3 -\eta) x^{36}z^{27} \\
& {}- (\eta^{18} + \eta^{12} + \eta^{10} + \eta^4 - \eta^2 - 1)x^{54}z^{27} - x^{90} z^{27}\\
& {}+ (\eta^{45} + \eta^{39} + \eta^{37} + \eta^{31} - \eta^{29} - \eta^{27} + \eta^{21}
+\eta^{19} - \eta^{13} + \eta^{11} - \eta^9 - \eta^7 \\
& {}+ \eta^5 - \eta^3 +\eta)x^{54}z^{81} - (\eta^{18} + \eta^{12} + \eta^{10} + \eta^4 - \eta^2
+1) x^{108}z^{81} \\
& {}- (\eta^{27} + \eta^9 + \eta^3 + \eta)x^{162}z^{81} + O(z^{243}).
\end{align*}
However, bounds on the degree of $z$ in $S_4^1(x,z)$ (see \cite[Prop.~IV.17(5)]{LutesThesis}) allow for non-zero $z^{243}$, $z^{729}$, $z^{2187}$, and $z^{6561}$ terms, and the required precision on the coefficients exceeded our computing power.  Fortunately, computations of $S_3^1(x,z)$ and $S_5^1(x,z)$ were unnecessary, since $S_m^1(x,1) = 0$ for all odd $m \geq 3$ (see \cite[Prop.~IV.17(6)]{LutesThesis}).

Combining the above computations with \eqref{E:explambda} (see also \cite[Ex.~VIII.4]{LutesThesis}), we find that
\begin{align}
L(1,\chi) &= \frac{\bigl(-\xi_1 + \eta(\theta-1)\xi_2 - (\xi_2 + \eta(\theta-1)\xi_1) \sqrt{-1} \bigr)}{(\theta+1)\theta\sqrt{\theta}} \cdot \tpi_1, \\
L(1,\chi^3) &= \frac{\bigl(-\xi_1 + \eta(\theta-1)\xi_2 + (\xi_2 + \eta(\theta-1)\xi_1) \sqrt{-1} \bigr)}{(\theta+1)\theta\sqrt{\theta}} \cdot \tpi_1, \quad (= \overline{L(1,\chi)}),\\
L(1,\chi^5) &= \frac{(-\sqrt{-1}\, \xi_1 + \xi_2)(\theta+1)}{\theta} \cdot \tpi_1, \\
L(1,\chi^7) &= \frac{(\sqrt{-1}\, \xi_1 + \xi_2) (\theta+1)}{\theta} \cdot \tpi_1, \quad (=
\overline{L(1,\chi^5)}).
\end{align}
Let $\Log_\rho(z)$ denote the multi-valued logarithm of $\rho$, for which $\log_\rho(z)$ is a particular branch, i.e.\ for $\alpha$, $\beta \in \KK$, we write $\Log_\rho(\alpha) = \beta$ if $\exp_\rho(\beta) = \alpha$.   Then we find
\begin{align}
  L(1,\chi^4) &= \frac{1}{\sqrt{\theta}} \Log_\rho \bigl( \theta^4 \sqrt{\theta} + (\eta
    + 1)\theta \sqrt{\theta} + \sqrt{\theta} \bigr),
    \label{E:L1chi4rho1} \\
  L(1,\chi^8) &= \biggl( 1 - \frac{1}{\theta} \biggr) \Log_\rho(\eta - 1), \label{E:L1chi8rho1}
\end{align}
where, since $\chi^8$ is the trivial character, \eqref{E:L1chi8rho1} is simply a restatement of \cite[Thm.~VI]{Thakur92}.  The reason for using $\Log_\rho$ is that the above arguments inside $\Log_\rho$ are not within the radius of convergence of $\log_\rho(z)$ according to \eqref{E:pibounds}, and we did not find any attractive formulas for $L(1,\chi^4)$ using $\log_\rho(z)$ only, as in \cite[Thm.~VI]{Thakur92}.  However, the right-hand side of \eqref{E:L1chi4rho1} is the unique logarithm of $\theta^4 \sqrt{\theta} + (\eta + 1)\theta \sqrt{\theta} + \sqrt{\theta}$ that is in $\laurent{\FF_3}{1/\sqrt{\theta}}$.  Finally, for $L(1,\chi^2)$ and $L(1,\chi^6)$ we have the formulas
\begin{multline}
 \biggl( \frac{\sqrt{\theta}}{\xi_1\xi_2} + \sqrt{-1}\,\xi_1\xi_2 \biggr)L(1,\chi^2) + (\theta+1)
   \sqrt{\theta}L(1,\chi^4) + \biggl( \frac{\sqrt{\theta}}{\xi_1\xi_2} - \sqrt{-1}\,\xi_1\xi_2
   \biggr)L(1,\chi^6) \\
   = \log_\rho \bigl( \xi_1^2 - (\eta^4 + \eta^2 - \eta)\xi_1^6 + \eta\xi_1^{12}
     - (\eta^9 - \eta^3 - \eta - 1) \xi_1^{18} + \xi_1^{36} + \xi_1^{54} \bigr)
\end{multline}
and
\begin{multline} \label{E:LsumS4}
 \biggl( \frac{\sqrt{\theta}}{\xi_1\xi_2} - \sqrt{-1}\,\xi_1\xi_2 \biggr)(\theta+1)\sqrt{\theta}L(1,\chi^2) + \eta(\theta-1)\sqrt{\theta}
   L(1,\chi^4)\\ {}+ \biggl( \frac{\sqrt{\theta}}{\xi_1\xi_2} + \sqrt{-1}\,\xi_1\xi_2
   \biggr) (\theta + 1)\sqrt{\theta} L(1,\chi^6) = -\Log_\rho(S_4^1(\xi_1,1)).
\end{multline}
Given that $|\xi_1| = 3^{-7/2}$, more than likely $S_4^1(\xi_1,1)$ is small enough so that the right-hand side of \eqref{E:LsumS4} is $-\log_\rho(S_4^1(\xi_1,1))$.
\end{example}


\begin{thebibliography}{99}
\bibitem{And94} G. W. Anderson, \textit{Rank one elliptic $A$-modules and $A$-harmonic series}, Duke Math. J. \textbf{73} (1994), 491--542.

\bibitem{And96} G. W. Anderson, \textit{Log-algebraicity of twisted
$A$-harmonic series and special values of $L$-series in
characteristic $p$}, J. Number Theory \textbf{60} (1996), 165--209.

\bibitem{CP} C.-Y. Chang and M. A. Papanikolas, \textit{Algebraic
independence of periods and logarithms of Drinfeld modules. With an appendix by
B.~Conrad}, arXiv:1005.5120, 2010.

\bibitem{Damamme99} G. Damamme, \textit{\'{E}tude de $L(s,\chi)/\pi^s$ pour des
fonctions $L$ relatives \`{a} $\mathbf{F}_q(\!(T^{-1})\!)$ et associ\'{e}es \`{a}
des caract\`{e}res de degr\'{e} $1$}, J. Th\'{e}or. Nombres Bordeaux
\textbf{11} (1999), 369--385.

\bibitem {Feng97} K. Feng, \textit{Anderson's root numbers and
Thakur's Gauss sums}, J. Number Theory \textbf{65} (1997), 279--294.

\bibitem {Goss} D. Goss, \textit{Basic structures of function field
    arithmetic}, Springer-Verlag, Berlin, 1996.

\bibitem{Hayes79}
D. R. Hayes, \textit{Explicit class field theory in global function fields},
in: Studies in algebra and number theory, Adv. in Math. Suppl. Stud. \textbf{6},
 Academic
Press, New York, 1979, pp.~173--217.

\bibitem{Hayes92}
D. R. Hayes, \textit{A brief introduction to Drinfeld modules}, in: The
Arithmetic of Function Fields (Columbus, OH, 1991), W. de Gruyter, Berlin, 1992,
 pp.~1--32.

\bibitem {LMQ75} J. R. C. Leitzel, M. L. Madan, and C. S. Queen,
\textit{Algebraic function fields with small class number}, J.
Number Theory \textbf{7} (1975), 11--27.

\bibitem {LutesThesis} B. A. Lutes, \textit{Special values of the
Goss $L$-function and special polynomials}, Ph.D. thesis, Texas~A\&M
University, 2010.  (Available at http://www.math.tamu.edu/$\sim$map/.)

\bibitem {MurtyMurty1} M. R. Murty and V. K. Murty, \textit{Transcendental
values of class group $L$-functions}, Math. Ann. (to appear).

\bibitem {MurtyMurty2} M. R. Murty and V. K. Murty, \textit{Transcendental
values of class group $L$-functions, II}, Proc. Amer. Math. Soc. (to appear).

\bibitem {Papanikolas08} M. A. Papanikolas, \textit{Tannakian duality for
Anderson-Drinfeld motives and algebraic independence of
   Carlitz logarithms}, Invent. Math. \textbf{171} (2008), 123--174.

\bibitem {Pellarin} F. Pellarin, \textit{$\tau$-recurrent sequences and
modular forms}, arXiv:1105.5819, 2011.

\bibitem {Taelman} L. Taelman, \textit{Special $L$-values of Drinfeld
modules}, Ann. of Math. (2) (to appear).

\bibitem {Thakur92} D. S. Thakur, \textit{Drinfeld modules and
    arithmetic in the function fields}, Internat. Math. Res. Notices
  \textbf{1992} (1992), no. 9, 185--197.

\bibitem {Thakur93} D. S. Thakur, \textit{Shtukas and Jacobi
sums}, Invent. Math. \textbf{111} (1993), 557--570.

\bibitem {Thakur} D. S. Thakur, \textit{Function field arithmetic},
World Scientific Publishing, River Edge, NJ, 2004.

\bibitem {Yu86} J. Yu, \textit{Transcendence and Drinfeld modules},
Invent. Math. \textbf{83} (1986), 507--517.

\bibitem {Yu97} J. Yu, \textit{Analytic homomorphisms into Drinfeld modules},
Ann. of Math. (2) \textbf{145} (1997), 215--233.

\bibitem {Zhao97} J. Zhao, \textit{On root numbers connected with
special values of $L$-functions over $\mathbb{F}_q(T)$}, J. Number
Theory \textbf{62} (1997), 307--321.

\end{thebibliography}
\end{document}